\documentclass[12pt]{amsart}
\usepackage{amscd,amsmath,amsthm,amssymb}
\usepackage{amsfonts,amssymb,amscd,amsmath,enumerate,verbatim}
\usepackage[left]{lineno}
\usepackage{pstricks, pst-plot,pst-3d}
\usepackage{tikz}
\newpsstyle{fatline}{linewidth=1.5pt}
\definecolor{verylight}{gray}{0.97}
\definecolor{light}{gray}{0.9}
\definecolor{medium}{gray}{0.85}
\definecolor{dark}{gray}{0.6}
\def\opn#1#2{\def#1{\operatorname{#2}}} 
\opn\chara{char} \opn\length{\ell} \opn\pd{pd} \opn\rk{rk}
\opn\projdim{proj\,dim} \opn\injdim{inj\,dim} \opn\rank{rank}
\opn\depth{depth} \opn\grade{grade} \opn\height{height}
\opn\embdim{emb\,dim} \opn\codim{codim}
\opn\Cl{Cl}

\opn\Tr{Tr} \opn\bigrank{big\,rank}
\opn\superheight{superheight}\opn\lcm{lcm}
\opn\trdeg{tr\,deg}
	\opn\reg{reg} \opn\lreg{lreg} \opn\ini{in} \opn\lpd{lpd}
	\opn\size{size} \opn\sdepth{sdepth}
	\opn\link{link}
    \opn\fdepth{fdepth}\opn\lex{lex}
	\opn\tr{tr}\opn\del{del}
	\opn\type{type}
	\opn\gap{gap}
	\opn\arithdeg{arith-deg}
	\opn\revlex{revlex}
	\opn\div{div} \opn\Div{Div} \opn\cl{cl} \opn\Cl{Cl}
	\opn\Spec{Spec} \opn\Supp{Supp} \opn\supp{supp} \opn\Sing{Sing}
	\opn\Ass{Ass} \opn\Min{Min}\opn\Mon{Mon}
	\opn\Ann{Ann} \opn\Rad{Rad} \opn\Soc{Soc}
	\opn\Im{Im} \opn\Ker{Ker} \opn\Coker{Coker} \opn\Am{Am}
	\opn\Hom{Hom} \opn\Tor{Tor} \opn\Ext{Ext} \opn\End{End}
	\opn\Aut{Aut} \opn\id{id}
	
	\opn\nat{nat}
	\opn\pff{pf}
	\opn\Pf{Pf} \opn\GL{GL} \opn\SL{SL} \opn\mod{mod} \opn\ord{ord}
	\opn\Gin{Gin} \opn\Hilb{Hilb}\opn\sort{sort}
	\opn\PF{PF}\opn\Ap{Ap}
	\opn\mult{mult}
	\opn\bight{bight}
	\opn\div{div}
	\opn\Div{Div}
	\opn\aff{aff}
	\opn\relint{relint} \opn\st{st}
	\opn\lk{lk} \opn\cn{cn} \opn\core{core} \opn\vol{vol}  \opn\inp{inp} \opn\nilpot{nilpot}
	\opn\link{link} \opn\star{star}\opn\lex{lex}\opn\set{set}
	\opn\width{wd}
	\opn\Fr{F}
	\opn\QF{QF}
	\opn\G{G}
	\opn\type{type}\opn\res{res}
	\opn\conv{conv}
	\opn\Deg{Deg}
	\opn\Sym{Sym}
	\opn\Con{Con}
	\opn\gr{gr}
	
	\def\pot#1#2{#1[\kern-0.28ex[#2]\kern-0.28ex]}

	\opn\dirlim{\underrightarrow{\lim}}
	\opn\inivlim{\underleftarrow{\lim}}
	\def\Implies{\ifmmode\Longrightarrow \else
		\unskip${}\Longrightarrow{}$\ignorespaces\fi}
	\def\implies{\ifmmode\Rightarrow \else
		\unskip${}\Rightarrow{}$\ignorespaces\fi}
	\def\iff{\ifmmode\Longleftrightarrow \else
		\unskip${}\Longleftrightarrow{}$\ignorespaces\fi}

	\let\:=\colon
	\newtheorem{Theorem}{Theorem}[section]
	\newtheorem{Lemma}[Theorem]{Lemma}

	\newtheorem{Example}[Theorem]{Example}

	\let\epsilon\varepsilon
	\let\kappa=\varkappa
	\def\qed{\ifhmode\textqed\fi
		\ifmmode\ifinner\quad\qedsymbol\else\dispqed\fi\fi}
	\def\textqed{\unskip\nobreak\penalty50
		\hskip2em\hbox{}\nobreak\hfil\qedsymbol
		\parfillskip=0pt \finalhyphendemerits=0}
	\def\dispqed{\rlap{\qquad\qedsymbol}}
	
	\opn\dis{dis}
	\def\pnt{{\raise0.5mm\hbox{\large\bf.}}}
	
	\opn\Lex{Lex}

\begin{document}
\title[Depth of edge ideals]{Depth of edge ideals and vertex connectivity of finite graphs}

\author[T.~Hibi]{Takayuki Hibi}
\author[S.~A.~ Seyed Fakhari]{Seyed Amin Seyed Fakhari}

\address{(Takayuki Hibi) Department of Pure and Applied Mathematics, Graduate School of Information Science and Technology, Osaka University, Suita, Osaka 565--0871, Japan}
\email{hibi@math.sci.osaka-u.ac.jp}
\address{(Seyed Amin Seyed Fakhari) Departamento de Matem\'aticas, Universidad de los Andes, Bogot\'a, Colombia}
\email{s.seyedfakhari@uniandes.edu.co}



\begin{abstract}
Let $G$ be a finite graph on $[n]:=\{1, \ldots, n\}$ and $\kappa(G)$ its vertex connectivity.  Let $S=K[x_1, \ldots, x_n]$ denote the polynomial ring in $n$ variables over a field $K$ and $I(G^c)$ the edge ideal of the complementary graph $G^c$ of $G$.  It is a classical result that $\depth{S/I(G^c)} \leq \kappa(G) + 1$.  We give a sharp lower bound of $\depth{S/I(G^c)}$ in terms of $n$ and $\kappa(G)$.  Furthermore, a sharp lower bound of $\depth{S/I(G^c)^2}$ as well as that of $\depth{S/I(G^c)^{(2)}}$ in terms of $n$ and $\kappa(G)$ is given.

\end{abstract}	
\maketitle
\thispagestyle{empty}

\section*{Introduction}
Every graph to be studied is simple and finite.  Let $G$ be a graph on the vertex set $V(G)=\{x_1, \ldots, x_n\}$ and $E(G)$ the set of edges of $G$.  For $W \subset V(G)$, $G - W$ stands for the induced subgraph $G|_{V(G) \setminus W}$ of $G$ on $V(G) \setminus W$. If $G$ is not a complete graph, then let $\kappa(G)$ denote the smallest cardinality of $W \subset V(G)$ for which $G - W$ is disconnected. We call $\kappa(G)$ the {\em vertex connectivity} of $G$. In particular, $G$ is connected if and only if $\kappa(G) > 0$.  
If $T$ is a tree on $[n]$, then $\kappa(G) = 1$.  If $C_n$ is a cycle of length $n$, then $\kappa(C_n) = 2$. We employ the convention that the vertex connectivity of the complete graph on $[n]$ is $n-1$.  

A {\em clique} of $G$ is a subset $C \subset V(G)$ for which $\{x_i,x_j\}\in E(G)$ for all $x_i,x_j \in C$ with $x_i \neq x_j$.  The {\em clique complex} of $G$ is the simplicial complex on $V(G)$ which consists of all cliques of $G$.  Let $S=K[x_1, \ldots, x_n]$ denote the polynomial ring in $n$ variables over a field $K$.  The {\em Stanley--Reisner ideal} of $\Delta(G)$ is the ideal  
\[
I_{\Delta(G)} = (x_i x_j : 1 \leq i < j \leq n, \{x_i,x_j\} \not\in E(G)) 
\]
of $S$ and the {\em Stanley--Reisner ring} of $\Delta(G)$ is the quotient ring $K[\Delta(G)]=S/I_{\Delta(G)}$.  The {\em edge ideal} of $G$ is the ideal 
\[
I(G) = (x_ix_j : 1 \leq i < j \leq n, \{x_i,x_j\} \in E(G)) \subset S.
\]
One has $I_{\Delta(G)} = I(G^c)$, where $G^c$ is the {\em complementary graph} (\cite[p.~153]{HHgtm260}) of $G$.  

It is observed in \cite{TH} that Hocster's formula \cite[Theorem 8.1.1]{HHgtm260} guarantees that
\begin{eqnarray}
\label{depth_kappa_G^c}   
\depth{S/I(G^c)} \leq \kappa(G) + 1.
\end{eqnarray}
If $G$ is a chordal graph (\cite[p.~155]{HHgtm260}), then the equality holds in $(\ref{depth_kappa_G^c})$, because $I(G^c)$ has a linear resolution \cite[Theorem 9.2.3]{HHgtm260}. Finding a sharp lower bound of $\depth{S/I(G^c)}$ in terms of $n$ and $\kappa(G)$ is one of our goals (Section $2$).  In addition, a sharp lower bound of depth of the second power $\depth{S/I(G^c)^2}$ together with that of the second symbolic power $\depth{S/I(G^c)^{(2)}}$ is studied (Sections $3$ and $4$).  

\section{Hocster's formula}
We recall what Hocster's formula is and how to prove the inequality (\ref{depth_kappa_G^c}).  Let $G$ be a graph on $V(G)=\{x_1, \ldots, x_n\}$ and $\Delta(G)$ the clique complex of $G$.  Hochster's formula says that the $(i,j)$-th graded Betti number of $K[\Delta(G)]$ is
\[
\beta_{i,j}(K[\Delta(G)]) = \sum_{|W|=j} \dim_K {\widetilde H}_{j-i-1}(\Delta(G)|_W;K), 
\]
where $$\Delta(G)|_W := \{ \sigma \in \Delta(G) : \sigma \subset W \} = \Delta(G) - (V(G) \setminus W)) .$$ In particular, 
\[
\beta_{i,i+1}(K[\Delta(G)]) = \sum_{|W|=i+1} \dim_K {\widetilde H}_{0}(\Delta(G)|_W;K). 
\]
In other words,
\[
\beta_{n-i,n-i+1}(K[\Delta(G)]) = \sum_{|W|=i-1} \dim_K {\widetilde H}_{0}(\Delta(G)|_{[n]\setminus W};K). 
\]
Since $\dim_K {\widetilde H}_{0}(\Delta(G)|_{[n]\setminus W};K)=0$ if and only if $G-W$ is connected, it follows that $\kappa(G) \geq i$ if and only if $\beta_{n-i,n-i+1}(K[\Delta(G)])=0$.  In other words, as is stated in \cite{TH} explicitly, one has
\[
\kappa(G) = \max\{\,i : \beta_{n-i,n-i+1}(K[\Delta(G)])=0\}.
\]
In particular,  
\[
\beta_{n-\kappa(G)-1,n-\kappa(G)}(K[\Delta(G)])\neq 0.
\]
It then follows that
\[
\projdim{K[\Delta(G)]} \geq n - \kappa(G) - 1.
\]
Now, Auslander--Buchsbaum formula guarantees that 
\[
n - \depth{K[\Delta(G)]} \geq n - (\kappa(G) + 1).
\]
In other words,
\begin{eqnarray*}
\label{depth_kappa}   
\depth{K[\Delta(G)]} - 1 \leq \kappa(G).
\end{eqnarray*}
Since $I_{\Delta(G)} = I(G^c)$, the inequality (\ref{depth_kappa_G^c}) follows.

\begin{Example}
{\em
Let $G$ be a chordal graph on $[n]$.  Since $I(G^c)$ has a linear resolution \cite[Theorem 9.2.3]{HHgtm260}, it follows that $\projdim{S/I(G^c)} = n - \kappa(G) - 1$.  One has $\depth{S/I(G^c)} - 1 = \kappa(G)$.  See \cite{HaHibi} for a related work.
}
\end{Example}

\begin{Example}
{\em
Let $G$ be the non-chordal graph of Figure $1$.  One has $\kappa(G) = 4$.  Since $S/I(G^c) = K[x_1, \ldots, x_6]/(x_2x_5, x_3x_6)$, one has $\depth{S/I(G^c)} = 4.$
}
\end{Example}

\begin{figure}
\centering
\begin{tikzpicture}[scale=2]
\node[draw,shape=circle] (1) at (1,1.5){};
\node[draw,shape=circle] (2) at (0,1){};
\node[draw,shape=circle] (3) at (0,0){};
\node[draw,shape=circle] (4) at (1,-0.5){};
\node[draw,shape=circle] (5) at (2,0){};
\node[draw,shape=circle] (6) at (2,1){};
\draw (1)--(2)--(3)--(4)--(5)--(6)--(1);
\draw (2)--(4)--(6)--(2);
\draw (3)--(5)--(1)--(3);
\draw (1)--(4);
\end{tikzpicture}
\caption{Non-chrdal graph with $\kappa(G) = \depth{S/I(G^c)} = 4$.}
\end{figure}
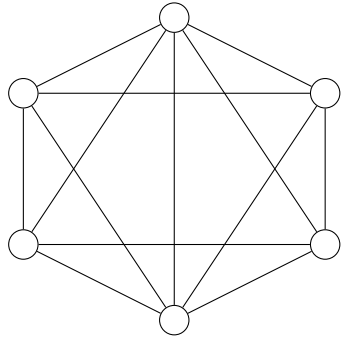

\begin{Example}
{\em
Let $G=K_{n,n,n}$ denote the complete multipartite graph (\cite[p.~394]{OH}) on $3n$ vertices.  Then $\depth{S/I(G^c)} = 3$ and $\kappa(G) = 2n$.   
}
\end{Example}

\section{A lower bound of depth of edge ideals}
Given a graph $G$ on $n$ vertices, a sharp lower bound of $\depth{S/I(G^c)}$ in terms of $n$ and $\kappa(G)$ is given.  Let $G$ be a graph on $V(G)=\{x_1, \ldots, x_n\}$.  The {\em neighbor set} of $x_i \in V(G)$ is $N_G(x_i):=\{x_j : \{x_i,x_j\}\in E(G)\}$.  In addition, $N_G[x_i]:=N_G(x_i)\cup \{x_i\}$.  Let $K_n$ be the complete graph on $n$ vertices.

\begin{Theorem} 
\label{main}
Let $G\neq K_n$ be a graph on $n$ vertices with $\kappa(G)=k$.  One has  $$\depth{S/I(G^c)}\geq\Big\lceil\frac{k}{2(n-k-1)}\Big\rceil+1.$$
\end{Theorem}

\begin{proof}
Let $V(G) = \{x_1, \ldots, x_n\}$ denote the vertex set of $G$.  We prove the inequality 
\[
\depth{S/I(G^c)}\geq\frac{k}{2(n-k-1)}+1.
\]
by using induction on $n$.  Since $G \neq K_n$, if $n=2$, then $G$ consists of two isolated vertices. So, $\kappa(G)=0$ and $\depth{S/I(G^c)}=1$. Let $n\geq 3$ and $\Delta = \Delta(G)$.  Suppose that there are $i$ and $W\subseteq V(G)$ for which ${\widetilde H}_i(\Delta|_W;K)\neq 0$.  By virtue of Hochster's formula, what we must prove is 
\begin{eqnarray}
\label{aaaaa}
|W|-i\leq n-\frac{k}{2(n-k-1)}.
\end{eqnarray}
If $W=\emptyset$, then $i=-1$ and (\ref{aaaaa}) is obvious. So, assume that $|W|\geq 1$. 

\medskip

\noindent
{\bf (Case 1.)}  Let $W\neq V(G)$.  Since ${\widetilde H}_i(\Delta|_W;K)\neq 0$, one has ${\widetilde H}_i(\Delta(H);K)\neq 0$, where $H:=G|_W$. In particular, $H$ is not a complete graph. It then follows from Hochser's formula that$$\depth{S_H}/I(H^c)\leq i+1,$$where $S_H:=K[x_i:x_i\in W]$. We deduce from induction hypothesis that$$\depth{S_H/I(H^c)}\geq\frac{\kappa(H)}{2(|W|-\kappa(H)-1)}+1.$$Consequently,$$\frac{\kappa(H)}{2(|W|-\kappa(H)-1)}\leq i.$$Furthermore, $$k=\kappa(G)\leq \kappa(H)+n-|W|.$$ Since $|W|-n< 0$, it follows that
\begin{align*}
i\geq\frac{k-n+|W|}{2(|W|-k+n-|W|-1)}=\frac{k-n+|W|}{2(n-k-1)}\geq \frac{k}{2(n-k-1)},
\end{align*}
as desired.

\medskip

\noindent
{\bf (Case 2.)} Let $W=V(G)$. So, ${\widetilde H}_i(\Delta;K)\neq 0$. Let $V(G) \setminus N_G[x_1]=\{x_{i_1}, \ldots, x_{i_t}\}$. Since ${\rm deg}_{G}(x_1)\geq k$, one has $k\leq n-t-1$.  If $i=0$, then ${\widetilde H}_0(\Delta;K)\neq 0$ and $G$ is disconnected. So, $k=0$ and (\ref{aaaaa}) is obvious. Let $i\geq 1$.  It follows that 
\[
\link_{\Delta}(\{x_{i_1}\}):= \{\sigma \in \Delta : x_{i_1} \not\in \sigma, \sigma \cup \{x_{i_1}\} \in \Delta\}
\]
is the clique complex of $H_1$, where $H_1$ is the induced subgraph of $G$ on $N_G(x_{i_1})$. 

Suppose that ${\widetilde H}_{i-1}(\link_{\Delta}(\{x_{i_1}\});K)\neq 0$. In particular, $H_1$ is not a complete graph. Hochster's formula says that$$\depth{S_{H_1}/I(H_1^c)}\leq i,$$where $S_{H_1}=K[x_j : x_j\in V(H_1)]$.  We deduce from induction hypothesis that$$\depth{S_{H_1}/I(H_1^c)}\geq\frac{\kappa(H_1)}{2(|V(H_1)|-\kappa(H_1)-1)}+1.$$It then follows that $$i\geq\frac{\kappa(H_1)}{2(|V(H_1)|-\kappa(H_1)-1)}+1$$Since $$k=\kappa(G)\leq \kappa(H_1)+n-|V(H_1)|$$ and $$|V(H_1)|=\deg_{G}(x_{i_1})\geq k,$$ the computation done in (Case 1) yields 
\begin{align*}
i\geq\frac{k-n+|V(H_1)|}{2(n-k-1)}+1\geq \frac{k}{2(n-k-1)}.
\end{align*}
So, we are done.  

Suppose that ${\widetilde H}_{i-1}(\link_{\Delta}(\{x_{i_1}\});K)=0$. Since ${\widetilde H}_i(\Delta;K)\neq 0$, by virtue of Mayer--Vietoris sequence, it follows that ${\widetilde H}_i(\Gamma;K)\neq 0$, where  $\Gamma:=\Delta|_{V(G) \setminus \{x_{i_1}\}}$, which is the clique complex of $G-\{x_{i_1}\}$. If $t=1$, then $\Gamma$ is the clique complex of the induced subgraph of $G$ on $N_G[x_1]$. In particular, $\Gamma$ is a cone which contradicts ${\widetilde H}_i(\Gamma;K)\neq 0$. Therefore, $t\geq 2$. Let $H_2$ be the induced subgraph of $G-\{ x_{i_1}\}$ on $N_{G- \{x_{i_1}\}}(x_{i_2})$, whose clique complex is $\link_{\Gamma}(\{x_{i_2}\})$. Suppose that ${\widetilde H}_{i-1}(\link_{\Gamma}(\{x_{i_2}\});K)\neq 0$. In particular, $H_2$ is not a complete graph. Hochster's formula says that$$\depth{S_{H_2}/I(H_2^c)}\leq i,$$where $S_{H_2}=K[x_j : x_j\in V(H_2)]$. We deduce from induction hypothesis that$$\depth{S_{H_2}/I(H_2^c)}\geq\frac{\kappa(H_2)}{2(|V(H_2)|-\kappa(H_2)-1)}+1.$$It then follows that  $$i\geq\frac{\kappa(H_2)}{2(|V(H_2)|-\kappa(H_2)-1)}+1.$$It follows from $t\geq 2$ that $k\leq n-3$. Since $$k=\kappa(G)\leq \kappa(H_2)+n-|V(H_2)|$$ and $$|V(H_2)|\geq \deg_G(x_{i_2})-1\geq k-1,$$ 
the computation done in (Case 1) yields 
\begin{align*}
i\geq\frac{k-n+|V(H_2)|}{2(n-k-1)}+1\geq \frac{k}{2(n-k-1)}.
\end{align*}
So, we are done. Suppose that ${\widetilde H}_{i-1}(\link_{\Gamma}(\{x_{i_2}\});K)=0$. Since ${\widetilde H}_i(\Gamma;K)\neq 0$, it follows from Mayer--Vietoris sequence that 
\[
{\widetilde H}_i(\Delta|_{V(G) \setminus \{x_{i_1}, x_{i_2}\}};K)={\widetilde H}_i(\Gamma|_{(V(G)\setminus\{x_{i_1}\})\setminus\{x_{i_2}\}};K)\neq 0. 
\]
Repeating the above procedures and using $t\leq n-k-1$, we deduce that $${\widetilde H}_i(\Delta - \{x_{i_1}, \ldots, x_{i_{\ell}}\};K)\neq 0$$ for each $1 \leq \ell \leq t$. In particular, $${\widetilde H}_i(\Delta - \{x_{i_1}, \ldots, x_{i_t}\};K)\neq 0.$$  Finally, recall that $\{x_{i_1}, \ldots, x_{i_t}\}=V(G)\setminus N_G[x_1]$. Hence $\Delta- \{x_{i_1}, \ldots, x_{i_t}\}$ is the clique complex of the induced subgraph of $G$ on $N_G[x_1]$. Since $\Delta- \{x_{i_1}, \ldots, x_{i_t}\}$ is a cone, one has ${\widetilde H}_i(\Delta - \{x_{i_1}, \ldots, x_{i_t}\};K)=0$, a contradiction.
\, \, \, \, \, \, \, \, \, \, \, 
\end{proof}

One of the dirsct consequences of Theorem \ref{main} is that if $\depth{S/I(G^c)} = 2$, then 
\begin{eqnarray}
\label{abcde}
k \leq \Big\lfloor\frac{\,2n-2\,}{3}\Big\rfloor.
\end{eqnarray}
It would be of interest to classify all pairs $(n,k)$ for which there is a graph $G$ on $n$ vertices with $k=\kappa(G)$ and $\depth{S/I(G^c)}=2$.

\begin{Example}
    {\em
Let $n \geq 5$.  Let $c_n$ denote the cycle of length $n$ on $\{x_1,\ldots,x_n\}$ and $C_n'$ that on $\{x_{1'},\ldots,x_{n'}\}$.  Let $G$ be the graph on $\{x_1, \ldots, x_{n}, x_{1'}\ldots, x_{n'}\}$ obtained from the disjoint union of $C_n \cup C_n'$ by adding those edges $\{x_i, x_j'\}$ with $i \neq j$.  One easily see that $k=\kappa(G)= n+1$ and $\depth{S/I(G^c)}=2$.  For example, if $n=10$, then $k = 6$.  Thus one has the equality in (\ref{abcde}). 

On the other hand, assume that $G=K_{5,5}$ is a complete bipartite graph.  Then $k=\kappa(G)= 5$ and $\depth{S/I(G^c)}=2$.  By removing edges of $K_{5,5}$, one easily construct a graph $H_i$ on $10$ vertices with $k=\kappa(G)= i$, where $i = 1,2,3,4$, and $\depth{S/I(G^c)}=2$.
    }
\end{Example}

\section{Depth of second symbolic powers of edge ideals}
We now turn to the study on a lower bound of depth of the quotient ring of the second symbolic power $I(G^c)^{(2)}$ of an edge ideal $I(G^c)$.

\begin{Theorem} \label{main2}
Let $G\neq K_n$ be a graph on $n$ vertices with $\kappa(G)=k$.  One has $$\depth{S/I(G^c)^{(2)}}\geq\Big\lceil\frac{k}{2(n-k-1)}\Big\rceil.$$
\end{Theorem}

\begin{proof}
Parts of the proof are similar to that of \cite[Theorem 4.2]{S}. So, we only sketch the proof with highlighting the difference.  We show the inequality$$\depth{S/I(G^c)^{(2)}}\geq\frac{k}{2(n-k-1)}.$$By virtue of the proof of \cite[Theorem 4.2]{S}, we show that for each $\{x_i,x_j\}\in E(G^c)$ and for each $A\subseteq N_{G^c}(x_i) \cup N_{G^c}(x_j)$ with $x_i,x_j\notin A$, one has
\begin{itemize}
\item [(i)] $\depth{S_A/(I(G^c\setminus A):x_i)}\geq \frac{k}{2(n-k-1)}$; and
\item [(ii)] $\depth{S_A/((I(G^c\setminus A):x_i)+(I(G^c\setminus A):x_j))}\geq \frac{k}{2(n-k-1)}-1$,
\end{itemize}
where $S_A=K[x_k : x_k\notin A]$. 

First we prove (i). Since $x_i\notin A$, one has $$\depth{S_A/(I(G^c\setminus A):x_i)}\geq 1.$$Thus (i) is clear if $k\leq (2n-2)/3$. Let $k> (2n-2)/3$.  It follows from the proof of \cite[Theorem 4.2]{S} that$$\depth{S_A/(I(G^c- A):x_i)}=\depth{S'/(I(G^c- (A\cup N_{G^c}[x_i])})+1,$$where $S'=K[x_k : x_k\notin A\cup N_{G^c}[x_i]]$.  Now, set $H:=G- (A\cup N_{G^c}[x_i])$. Then $G^c- (A\cup N_{G^c}[x_i])=H^c$.

Suppose that either $H$ is a complete graph or $V(H)=\emptyset$. One has $$A\cup N_{G^c}[x_i]\subseteq N_{G^c}(x_i)\cup N_{G^c}(x_j).$$ Since each vertex of $G$ is contained at least $k$ edges, each vertex of $G^c$ is contained at most $n-k-1$ edges. Thus$$|A\cup N_{G^c}[x_i]|\leq |N_{G^c}(x_i)\cup N_{G^c}(x_j)|\leq 2(n-k-1).$$Consequently, $|V(H)|\geq 2k-n+2$. Since $H$ is either a complete graph or $V(H)=\emptyset$ and since $k> (2n-2)/3$, it follows that
\begin{align*}
&\depth{S'/(I(G^c- (A\cup N_{G^c}[x_i])})+1 \\=& \, \depth{S'/I(H^c)}+1 \, = \, |V(H)|+1\\ \geq & \, \, 2k-n+3 \, \geq \, \frac{\,\,k\,\,}{2} \, \geq \, \frac{k}{\,2(n-k-1)\,}.  
\end{align*}

Suppose that $V(H)\neq \emptyset$ and $H$ is not a complete graph. Theorem \ref{main} says that 
\begin{align*}
& \depth{S'/(I(G^c\setminus (A\cup N_{G^c}[x_i])})+1\\=&  \depth{S'/I(H^c)}+1 \\ \geq & \,\, \frac{\kappa(H)}{2(|V(H)-\kappa(H)-1)}+1 \\\geq& \, \, \frac{k-n+|V(H)|}{2(n-k-1)}+1 \\  \geq & \, \, \frac{k}{2(n-k-1)} \quad (\, \text{since} \, \, \,  |V(H)|\geq 2k-n+2).
\end{align*}
This completes the proof of (i). 

The proof of (ii) is similar.
\, \, \, \, \, \, \, \, \, \, \, \, \, \, \, \, \, \, \, \, \, \, \, \, \, \, \, \, \, \, \, \, \, \, 
\end{proof}

\section{Depth of second powers of edge ideals}
Finally, a lower bound of depth of the quotient ring of the second power $I(G^c)^{2}$ of an edge ideal $I(G^c)$ is studied.

\begin{Lemma} \label{ineq}
Let $k$ and $n$ be integers with $0\leq k\leq n-2$ and $k> (2n-2)/3$.  One has $$3k-2n+3\geq \frac{k}{\,2(n-k-1)\,}-1.$$    
\end{Lemma}
\begin{proof}
If $k\leq (6n-6)/7$, then$$\frac{k}{2(n-k-1)}\leq 3$$and the assertion follows from $3k-2n+3\geq 2$. If $k> (6n-6)/7$, then$$3k-2n+3\geq \frac{\,\,k\,\,}{2}\geq \frac{k}{\,2(n-k-1)\,},$$ 
as desired.
\hspace{11.6cm}
\end{proof}
\begin{Theorem} \label{main3}
Let $G\neq K_n$ be a graph on $n$ vertices with $\kappa(G)=k$. Then$$\depth{S/I(G^c)^2}\geq\Big\lceil\frac{k}{2(n-k-1)}\Big\rceil-1.$$
\end{Theorem}

\begin{proof}
Parts of the proof are similar to that of \cite[Theorem 3.6]{S'}. We only sketch the proof with highlighting the difference.  We show the inequality
\begin{eqnarray}
    \label{bbbbb}
\depth{S/I(G^c)^2}\geq\frac{k}{2(n-k-1)}-1.
\end{eqnarray}
If $k\leq  (2n-2)/3$, then $\displaystyle\frac{k}{2(n-k-1)} \leq 1$ and (\ref{bbbbb}) is trivial. Let $k>(2n-2)/3$.

By virtue of \cite[Theorem 3.6]{S'}, we show that for each $\{x_i,x_j\}\in E(G^c)$ and for each $A\subseteq N_{G^c}(x_i)\cup N_{G^c}(x_j)$ with $x_i,x_j\notin A$, one has
\begin{eqnarray}
\label{ccccc}
    \depth{S_A/(I(G^c\setminus A)^2:x_ix_j)}\geq \frac{k}{2(n-k-1)}-1,
\end{eqnarray}
where $S_A:=K[x_k : x_k\notin A]$.

Set $r:=|N_{G^c}(x_i)\cup N_{G^c}(x_j)\setminus\{x_i,x_j\}|$. In particular $|A|\leq r$. We prove (\ref{ccccc}) by backward induction on $|A|$.  Since each vertex of $G$ is contained at least $k$ edges, each vertex of $G^c$ is contained
at most $n - k - 1$ edges.  Since $x_i,x_j\not\in A$, it follows that $|A|\leq 2(n-k-2)$. 

Suppose that $|A|=r$. Then $G^c- A$ is the disjoint union of $G^c- (N_{G^c}(x_i)\cup N_{G^c}(x_j))$ and the edge $\{x_i,x_j\}$. Hence, $$(I(G^c\setminus A)^2:x_ix_j)=I(G^c- A).$$
Consequently, Theorem \ref{main} implies that
\begin{eqnarray}
\label{ddddd}    
\depth{S_A/(I(G^c- A)^2:x_ix_j)}\geq \frac{\kappa(G\setminus A)}{2(|V(G\setminus A)-\kappa(G\setminus A)-1)}+1.
\end{eqnarray}
Since 
\[
k=\kappa(G)\leq \kappa(G\setminus A)+|A|\leq \kappa(G\setminus A)+2(n-k-2), 
\]
the above inequality (\ref{ddddd}) implies that$$\depth{S_A/(I(G^c\setminus A)^2:x_ix_j)}\geq \frac{k}{2(n-k-1)}.$$

Suppose that $|A|< r$. Set $$L:=N_{G^c\setminus A}(x_i) \cap N_{G^c\setminus A}(x_j).$$ It follows from \cite[Theorems 6.5 and 6.7]{b} that
\begin{align*}
(I(G^c- A)^2:x_ix_j)\, = & \, I(G^c- A)\\ +& \, \, (x_px_q : x_p\in N_{G^c- A}(x_i), x_q\in N_{G^c- A}(x_j),p\neq q)\\  +& \, \, (x_k^2 : x_k\in L).
\end{align*}
As a consequence, there is a graph $H$ for which $I(H)$ is the polarization (\cite[p.~19]{HHgtm260}) of $(I(G^c\setminus A)^2:x_ix_j)$. Let $$V(H)=\{x_t\mid x_t\in V(G\setminus A\}\cup\{y_k|x_k\in L\}.$$ Also, let $T$ be the polynomial ring over $K$ whose variables correspond to the vertices of $H$. 
We know from \cite[Corollary 1.6.3]{HHgtm260} that$$\depth{S_A/(I(G^c\setminus A)^2:x_ix_j)}=\depth{T/I(H)}-|L|.$$

First, suppose that $L=\emptyset$. It follows from \cite[Lemma 3.2]{S3} and \cite[Corollary 1.3]{R} that
\begin{align*}
& \, \depth{S_A/(I(G^c- A)^2:x_ix_j)}\\ = &\, \depth{S_A/(I(G^c- A)^{(2)}:x_ix_j)}\\  \geq & \, \depth{S_A/I(G^c- A)^{(2)}},
\end{align*}
Since $x_ix_j\in E(G^c- A)$, we conclude that $G- A$ is not a complete graph. Hence, it follows from Theorem \ref{main2} and the above inequality that
\begin{eqnarray}
\label{eeeee}   
\depth{S_A/(I(G^c- A)^2:x_ix_j)}\geq \frac{\kappa(G\setminus A)}{2(|V(G\setminus A)|-\kappa(G\setminus A)-1)}.
\end{eqnarray}
Since $$k=\kappa(G)\leq \kappa(G\setminus A)+|A|\leq \kappa(G\setminus A)+2(n-k-2),$$ it follows from (\ref{eeeee}) that$$\depth{S_A/(I(G^c- A)^2:x_ix_j)}\geq \frac{k}{2(n-k-1)}-1.$$

Second, suppose that $L\neq \emptyset$ and choose a vertex $x_{\ell}\in L$. It follows from \cite[Lemma 5.1]{dhs} that$$\depth{T/I(H)}\in\{\depth{T/(I(H),x_{\ell})}, \depth{T/(I(H):x_{\ell})}\}.$$So, the following two cases arise.

\medskip

\noindent
{\bf (Case 1.)} Suppose that $\depth{T/I(H)}=\depth{T/(I(H),x_{\ell})}$. Since $$(I(H),x_{\ell})=(I(H- x_{\ell}),x_{\ell}),$$ it follows that $y_{\ell}$ is a regular element on $T/(I(H),x_{\ell})$. Furthermore, $I(H\setminus x_{\ell})$ is the polarization of $(I(G\setminus (A\cup \{x_{\ell}\}))^2:x_ix_j)$. It then follows  from \cite[Corollary 1.6.3]{HHgtm260} and our backward induction on $|A|$ that
\begin{align*}
& \depth{S_A/(I(G^c\setminus A)^2:x_ix_j)}\\ = & \, \depth{T/I(H)}-|L|\\ =& \, \depth{T/(I(H\setminus x_{\ell}), x_{\ell})}-|L|\\  \geq & \, \frac{k}{2(n-k-1)}-1+|L|-1+1-|L|\\  = & \, \frac{k}{2(n-k-1)}-1,  
\end{align*}
as desired.

\medskip

\noindent
{\bf (Case 2.)} Suppose that $\depth{T/I(H)}=\depth{T/(I(H): x_{\ell})}$. Set$$G':=G- (N_{G^c}(x_i)\cup N_{G^c}(x_j)\cup N_{G^c}(x_{\ell})).$$It follows from the structure of $H$ that$$(I(H):x_{\ell})=I(G'^c)+(x_k : x_k\in N_H(x_{\ell})).$$ In particular, the variables in the set $\{y_{\ell'} : x_{\ell'} \in L, \ell\neq \ell'\}\cup\{x_{\ell}\}$ form a regular sequence on $T/I(H)$. Set $S':=K[x_k : x_k\in V(G')]$. One has$$\depth{T/I(H)}=\depth{S'/I(G'^c)}+|L|.$$Recall that each vertex of $G^c$ belongs to at most $n-k-1$ edges. Therefore, $$|(N_{G^c}(x_i)\cup N_{G^c}(x_j)\cup N_{G^c}(x_{\ell})|\leq 3(n-k-1).$$As a consequence, $|V(G')|\geq 3k-2n+3$. If $G'$ is a complete graph, then 
\begin{align*}
& \depth{S_A/(I(G^c\setminus A)^2:x_ix_j)} \\ =& \,\depth{T/I(H)}-|L|\\  = & \, |V(G')|+|L|-|L|=|V(G')|\geq 3k-2n+3\\ \geq & \,\frac{k}{2(n-k-1)}-1,
\end{align*}
where the last inequality follows from Lemma \ref{ineq}.
If $G'$ is not a complete graph, then we deduce from Theorem \ref{main} that
\begin{align*}
& \depth{S_A/(I(G^c\setminus A)^2:x_ix_j)}\\= & \, \depth{T/I(H)}-|L|\\  \geq & \, \frac{\kappa(G')}{2(|V(G')-\kappa(G')-1)}+1+|L|-|L|\\  = & \, \frac{\kappa(G')}{2(|V(G')-\kappa(G')-1)}+1.  
\end{align*}
Furthermore, since $$k=\kappa(G)\leq \kappa(G')+|V(G)|-|V(G')|\leq \kappa(G')+3(n-k-1),$$one has$$ \depth{S_A/(I(G^c\setminus  A)^2:x_ix_j)}\geq\frac{k}{2(n-k-1)}-1,$$ 
as desired.
\hspace{11.7cm}
\end{proof}

\begin{Example}
{\em 
Let $G=C_6$.  One has $\depth{S/I(G^c)}=2, \depth{S/I(G^c)^{(2)}}=1$, and $\depth{S/I(G^c)^{2}}=0$.  Thus, the equality holds in each of the inequalities of Theorems \ref{main}, \ref{main2} and \ref{main3}.
}
\end{Example}

\section*{Acknowledgments}
The second author is supported by a FAPA grant from Universidad de los Andes.

\section*{Statements and Declarations}
The authors have no Conflict of interest to declare that are relevant to the content of this article.

\section*{Data availability}
Data sharing does not apply to this article as no new data were created or analyzed in this study.


\begin{thebibliography}{99}
\bibitem{b}
A.~Banerjee, The regularity of powers of edge ideals, {\em J. Algebraic Combin.} {\bf 41} (2015), 303–321.
\bibitem{dhs}
H.~Dao, C.~Huneke, J.!Schweig, Bounds on the regularity and projective dimension of ideals associated to graphs, {\em J. Algebraic Combin.} {\bf 38} (2013), 37--55.
\bibitem{HHgtm260} 
J.~Herzog and T.~Hibi, ``Monomial Ideals,'' GTM 260, Springer, 2011.
\bibitem{HaHibi}
T.~H.~H\`a and T.~Hibi, Vertex connectivity of chordal graphs, {\em Discrete Math.} {\bf 349} (2026), \#114777.
\bibitem{H} 
T.~Hibi, The comparability graph of a modular lattice, {\em Combinatorica} {\bf 18} (1998), 541--548.
\bibitem{OH}
H.~Ohsugi and T.~Hibi, Compressed polytopes, initial ideals and complete multipartite graphs, {\em Illinois J. Math.} {\bf 44} (2000), 391--406.
\bibitem{R}
A.~Rauf, Depth and Stanley depth of multigraded modules, {\em Comm. Algebra}  {\bf 38} (2010), 773--784.
\bibitem{S}
S.~A.~Seyed Fakhari, On the depth of symbolic powers of edge ideals of graphs, {\em Nagoya Math. J.} {\bf 245} (2022), 28--40.
\bibitem{S3}
S.~A.~Seyed Fakhari, On the regularity of small symbolic powers of edge ideals of graphs, {\em Math. Scand.}, {\bf 129} (2023), 39--59.
\bibitem{S'}
S.~A.~Seyed Fakhari, Lower bounds for the depth of the second power of edge ideals, {\em Collect. Math.} {\bf 75} (2024), 535--544.
\bibitem{TH}
N.~Terai and T.~Hibi, Finite free resolutions and $1$-skeletons of simplicial complexes, {\em J. Alg. Combin.} {\bf 6} (1997), 89--93.
\end{thebibliography}
\end{document}